\documentclass[a4paper, 12pt]{amsart}
\usepackage{amscd,amsthm,amsfonts,latexsym,amssymb}
\usepackage[left=2.5cm, right=2.5cm, top=4cm]{geometry}

\usepackage{amsfonts}
\usepackage{amsthm}
\usepackage{graphicx}

\usepackage{multicol}

\usepackage{tikz}

\usepackage{graphicx, amssymb, latexsym, amsfonts, amsmath, lscape, amscd,
amsthm, color, epsfig}

\usepackage[T1]{fontenc}
\usepackage[latin1]{inputenc}

\usepackage[babel=true]{csquotes}

\usepackage{setspace}
\usepackage{verbatim}
\usepackage{color}
\usepackage{listings}
\definecolor{Code}{rgb}{0,0,0}
\definecolor{Decorators}{rgb}{0.5,0.5,0.5}
\definecolor{Numbers}{rgb}{0.5,0,0}
\definecolor{MatchingBrackets}{rgb}{0.25,0.5,0.5}
\definecolor{Keywords}{rgb}{0,0,1}
\definecolor{self}{rgb}{0,0,0}
\definecolor{Strings}{rgb}{0,0.63,0}
\definecolor{Comments}{rgb}{0,0.63,1}
\definecolor{Backquotes}{rgb}{0,0,0}
\definecolor{Classname}{rgb}{0,0,0}
\definecolor{FunctionName}{rgb}{0,0,0}
\definecolor{Operators}{rgb}{0,0,0}
\definecolor{Background}{rgb}{0.98,0.98,0.98}
\lstdefinelanguage{Python}{
aboveskip=\bigskipamount,
belowskip=\bigskipamount,
frame=l,
numbers=left,
numberstyle=\footnotesize,
numbersep=1em,
xleftmargin=1em,
framextopmargin=2em,
framexbottommargin=2em,
showspaces=false,
showtabs=false,
showstringspaces=false,
tabsize=2,
basicstyle=\ttfamily\small\setstretch{1},
backgroundcolor=\color{Background},
comment=[l]{\#},
commentstyle=\color{Comments}\slshape,
stringstyle=\color{Strings},
morestring=[s][\color{Strings}]{"}{"},
morestring=[s][\color{Strings}]{'}{'},
morecomment=[s][\color{Strings}]{"""}{"""},
morecomment=[s][\color{Strings}]{'''}{'''},
morekeywords={import,from,class,def,for,while,if,is,in,elif,else,not,and,or,print,break,continue,return,True,False,None,access,as,del,except,exec,finally,global,import,lambda,pass,print,raise,try,assert,yield},
keywordstyle={\color{Keywords}\bfseries},
morekeywords={[2]@invariant,pylab,numpy,np,scipy},
keywordstyle={[2]\color{Decorators}\slshape},
emph={self},
emphstyle={\color{self}\slshape},
escapeinside={\%*}{*)},
}

\newtheorem{theorem}{Theorem}
\newtheorem{theorem*}{Theorem}

\newtheorem{conjecture}[theorem]{Conjecture}

\newtheorem{lemma}[theorem]{Lemma}
\newtheorem{proposition}[theorem]{Proposition}

\newtheorem{remark}[theorem]{Remark}
\author{{Amine El Sahili}$^1$ \and {Zeina Ghazo Hanna}$^1$}
\address{Lebanese University, KALMA Laboratory, Beirut, Lebanon}
\email{sahili@ul.edu.lb ; zeina$\_$hanna$\_$93@live.com}

\title[Oriented Hamiltonian Paths]{Counting Hamiltonian Paths in Transitive Tournaments}

\begin{document}

\maketitle
\begin{abstract}
We construct a combinatorial function $\mathcal{F}$ which computes the number of oriented Hamiltonian paths of any given type, in a transitive tournament. We also study many properties of $\mathcal{F}$ that arise, and reach some observations.
\end{abstract}

\section{Introduction}

Let $T$ be a tournament of order $n$. An oriented path $P$ of $T$ is a subdigraph of $T$ whose underlying graph is a path. $P$ is a directed path if it is an oriented path whose all arcs have the same direction. An antidirected path $P$ in $T$ is an oriented path of $T$ whose arcs have successively opposite directions. A Hamiltonian oriented path $P$ in $T$ is an oriented path such that $V(P)=V(T)$.\\

\par Counting Hamiltonian paths in a tournament is a widely treated topic. Given a certain type of oriented Hamiltonian paths, (that is, fixing a random orientation of their arcs), one may ask how many such paths can be found in a tournament. No exact value of these numbers was given. What was done in this area is bounding the number of only the directed Hamiltonian paths  in tournaments. The oldest result through this investigation was given by Szele \cite{Szele}, who gave lower and upper bounds for the maximum number $P(n)$ of Hamiltonian directed paths in a tournament on $n$ vertices, and which was considered to be an introduction to the probabilistic methods in graph theory: $$\frac{n!}{2^{n-1}}\leq P(n) \leq c_1\frac{n!}{2^{\frac{3}{4}n}},$$where $c_1$ is a positive constant independent of $n$. Later on, the upper bound of $P(n)$ was improved by Alon \cite{Alon}: $P(n)\leq c_2.n^{\frac{3}{2}}\frac{n!}{2^{n-1}}.$ To prove it, he used Minc's conjecture (proved by Bregman) and observed that the permanent of the adjacency matrix of a tournament and the number of 1-factors of the tournament are equal. For the minimum number of Hamiltonian directed paths in a tournament, it can be easily verified that it is equal to $1$, and this value corresponds to the transitive tournament $TT_n$, which is a tournament of order $n$ whose all arcs are forward with respect to some enumeration of its vertices. But in the case of strong tournaments, this number increases a lot, as for the nearly-transitive tournament of order $n$, which is a tournament obtained from a transitive tournament $TT_n$ by reversing the orientation of the arc whose ends are the extremities of the unique directed Hamiltonian path in $TT_n$. In the nearly-transitive tournament, the number of directed Hamiltonian paths is equal to $2^{n-2}+1$. In 1972, Moon \cite{Moon} gave upper and lower bounds for the minimum number $p(n)$ of directed Hamiltonian paths in a strong tournament of order $n$:
\begin{equation*}
\alpha^{n-1}\leq p(n)\leq
\left\lbrace
\begin{array}{ccc}
3 \cdot \beta^{n-3}  & \mbox{for} & n\equiv 0 \mod 3\\
\beta^{n-1} & \mbox{for} & n\equiv 1 \mod 3\\
9 \cdot \beta^{n-5} & \mbox{for} & n\equiv 2 \mod 3
\end{array}\right.
\end{equation*}
where $\alpha=6^{\frac{1}{4}}\approx 1.565$ and $\beta=5^{\frac{1}{3}}\approx 1.710$, and in 2006, after finding an interesting characterization of strong tournaments, Busch \cite{Busch} improved this result by proving that the exact value of this minimum number is equal to the upper bound given by Moon. \\
\linebreak
\noindent Concerning non-directed Hamiltonian paths, Rosenfeld \cite{Rosenfield2} proved in 1974 that the number of antidirected Hamiltonian paths starting with a forward arc is equal to the number of antidirected Hamiltonian paths starting with a backward arc, in any tournament, which can be stated as: the number of Hamiltonian antidirected paths in any tournament $T$ is equal to the number of antidirected Hamiltonian paths in the complement of $T$. In \cite{ESandGH}, this result was generalized, by proving that every tournament and its complement contain the same number of oriented Hamiltonian paths of any given type. %On the other hand, for some positive integer $n$, D\'esir\'e Andr\'e \cite{DA} studied in 1881 the notion of alternating permutations of $n$ distinct labelled letters. A permutation of $n$ letters labelled $\lbrace a_1, a_2, \dots, a_n \rbrace$ is a given order of these letters, and for a fixed permutation of the $a_i's$, one defines the sequence of relative integers $s_1, s_2, \dots, s_{n-1}$ such that $\forall 1\leq i \leq n-1$, $s_i=a_i-a_{i+1}$. If every two consecutives elements $s_i's$ of this sequence have opposite signs, then the corresponding permutation of the $a_i's$ is said to be alternating, and, after proving that the number of alternating permutations of $n$ distinct labelled letters is always even, D\'esir\'e Andr\'e denoted this number by $2An$. Interestingly, for the particular case of a transitive tournament $TT_n$, after labelling its vertices by $v_1, v_2, \dots, v_n$, such that all its arcs are forward with respect to this enumeration, it is easy to see that each antidirected Hamiltonian path in this tournament corresponds to exactly two alternating permutations of the $n$ indexes, since each path has two enumerations of its vertices. Based on this, in his paper, Rosenfeld \cite{Rosenfield2} proved that in a transitive tournament $TT_n$, the number of distinct antidirected Hamiltonian paths is equal to $An$ if $n$ is even, and $2An$ if $n$ is odd. Furthermore, 
Rosenfeld \cite{Rosenfield2} also proved that for any tournament $T$ of order $n$, where $T \neq TT_n$, the number of its antidirected Hamiltonian paths is less than the number of antidirected Hamiltonian paths in $TT_n$.\\

\par In this paper, we are interested in counting Hamiltonian paths in transitive tournaments. We construct a combinatorial function giving the exact number of oriented Hamiltonian paths of any given type in a transitive tournament, study some properties this function holds, and build a program to compute its values, which will yield to an interesting observation regarding the number of antidirected Hamiltonian paths in a transitive tournament.
\smallskip
\section{Basic definitions and preliminary results}

Let $\mathbb{K}_s = \lbrace (\alpha_1,\alpha_2,\dots,\alpha_s) \in \mathbb{Z}^s, \ s\geq 1, \ \alpha_i \cdot \alpha_{i+1}<0, \ \forall \ 1\leq i \leq s-1 \rbrace$.\\
\noindent Since the orientations of the arcs of an oriented path are arbitrary, one may give a more precise definition to an oriented path by assigning to each one a type, as defined in \cite{ElSahili-AA2}:\\
\linebreak
Let $(\alpha_1,\alpha_2,\dots,\alpha_s) \in \mathbb{K}_s$. An oriented path $P$ is of type $P(\alpha_1,\alpha_2,\dots,\alpha_s)$ if $P$ is formed by $s$ blocks (a block is a maximal directed path) $I_1, I_2, \dots, I_s$ such that $length(I_i) = \mid\! I_i \!\mid = \mid\! \alpha_i \!\mid$ and with $x_i, y_i$ being the ends of the block $I_i$, $I_i \cap I_{i+1} = \{y_i\} =
\{x_{i+1}\}$, %we choose $x_1$ and $y_1$ such that the following condition is verified: $\alpha_1 >0 \iff I_1 \text{ is directed from } x_1 \text{ to } y_1$, and so
the following condition is verified: $\forall \ i=1,\dots,s, \ \alpha_i>0 \iff I_i \text{ directed from } x_i \text{ to } y_i$. %We note $end(I_i)=\{x_i,y_i\}$, and
 We write $P = I_1 I_2 \dots I_s$. %For $u,v\in I_i$, $I_i[u,v]$ denotes the subpath of $I_i$ of ends $u$ and $v$. \\
If the oriented path is Hamiltonian in a tournament $T$ of order $n$, we have $\sum_{i=1}^s |\alpha_i|=n-1$.\\
\linebreak
\noindent This notation can be extended by allowing $\alpha_i$ to be $0$, by considering $P(0,\alpha_2,...,\alpha_s)=P(\alpha_2,...,\alpha_s)$, and $P(\alpha_1,...,\alpha_{s-1},0)=P(\alpha_1,...,\alpha_{s-1})$, and $P(\alpha_1, ..., \alpha_i, 0, \alpha_{i+2}, ..., \alpha_s) = P(\alpha_1,...,\alpha_i+\alpha_{i+2},...,\alpha_s)$ (remark that in this case, $\alpha_i$ and $\alpha_{i+2}$ have the same sign). This is useful for latter calculations.\\
\linebreak
We denote by $\mathcal{P}_T(\alpha_1,\alpha_2,\dots,\alpha_s)$ the set of oriented paths in $T$ of type $P(\alpha_1,\alpha_2,\dots,\alpha_s)$, and $f_T(\alpha_1,\alpha_2,\dots,\alpha_s)$ denotes the cardinal of this set, i.e. the number of oriented paths of type $P(\alpha_1,\alpha_2,\dots,\alpha_s)$ in $T$. Particularly $f_T(n-1)$ is the number of directed Hamiltonian paths in a tournament $T$ on $n$ vertices.\\ Note that the sets  $\mathcal{P}_T(\alpha)$, $\alpha=(\alpha_1,
\alpha_2,\dots,\alpha_s) \in \mathbb{K}_s$, $\sum\limits_{i=1}^s \mid \alpha_i \mid =n-1$, form a partition of the set $\mathcal{P}_T$ of all the Hamiltonian oriented paths in a tournament $T$ of order $n$.\\
\linebreak
For $\alpha = (\alpha_1,\dots,\alpha_s) \in \mathbb{Z}^s$, we denote by $- \alpha$ the tuple $(-\alpha_1,\dots,-\alpha_s)$ and by
$\overline{\alpha}$ the tuple $(\alpha_s,\alpha_{s-1}\dots,\alpha_1)$.\\
\noindent Two paths are equal if they have the same set of arcs. Since every oriented path can have at most two types, depending on which one of its ends we choose to start its enumeration, we have the following equivalence:
$$\mathcal{P}_T(\alpha) = \mathcal{P}_T(\beta) \iff \alpha = \beta \text{ or } \alpha = -\overline{\beta}. \ \ (*)$$
\linebreak
\noindent Let $\alpha = (\alpha_1,\dots,\alpha_s) \in \mathbb{Z}^s$. For some $1\leq i \leq s$, $\alpha^i$ denotes the tuple $(\alpha_1,\alpha_2,\dots,\alpha_i*1,\alpha_{i+1},\dots,\alpha_s)$, where $\alpha_i*1=\alpha_i-1$ if $\alpha_i>0$, and $\alpha_i*1=\alpha_i+1$ otherwise.

\begin{proposition}\label{PP'} We have:
$$\mathcal{P}_T(\alpha^i)=\mathcal{P}_T(\alpha^j)\iff i=j \ or \ \alpha^j=-\overline{\alpha^i}.$$
\end{proposition}
\begin{proof} The proof of Proposition \ref{PP'} is a direct consequence of the property $(*)$.
\end{proof}
Let $\alpha = (\alpha_1,\dots,\alpha_s) \in \mathbb{Z}^s$. The tuple $\alpha$ is symmetric if $\alpha=-\overline{\alpha}$. An oriented path $P$ in a tournament $T$ is symmetric if $P$ is of type $P(\alpha)$, where $\alpha$ is symmetric.\\

\par Let $T$ be a tournament of order $n$. $T$ is transitive, denoted by $TT_n$, if there exists an enumeration of its vertices such that $V(T)=\lbrace v_1,v_2,\dots ,v_n\rbrace $ and $E(T)=\lbrace (v_i,v_{j});i<j \rbrace$.\\
\linebreak
\noindent The number of directed Hamiltonian paths in $TT_n$ is equal to $1$, that is $$f_{TT_n}(n-1)=1.$$In fact, let $P=v_1v_2...v_n$ be a Hamiltonian path of $TT_n$, and suppose that $TT_n$ contains another Hamiltonian path $P'=v_{i_1}v_{i_2}...v_{i_n}$. Since $TT_n$ is transitive, $(v_{i_{k-1}},v_{i_k}) \in E(T) \Rightarrow i_{k-1} < i_k$, thus $1\leqslant i_1<i_2<...<i_{n-1}<i_n\leqslant n$ and therefore $i_j=j$ $\forall 1 \leq j \leq n$, so $P'=v_1v_2...v_n$ and we conclude that the number of directed Hamiltonian paths in $TT_n$ is equal to 1.\\
\linebreak
\noindent Also, since a transitive tournament $TT_n$ and its complement are isomorphic, then $\forall$ $\alpha=(\alpha_1,\alpha_2,\dots,\alpha_s) \in \mathbb{K}_s$ such that $\sum_{i=1}^s |\alpha_i|=n-1$, we have: $$f_{TT_n}(\alpha)=f_{TT_n}(-\alpha).$$
Moreover, one can easily compute the number of oriented Hamiltonian paths with two blocks in $TT_n$:
\begin{proposition}\label{combination}
Let $TT_n$ be the transitive tournament on $n$ vertices, and $(\alpha_1,\alpha_2) \in \mathbb{K}_2$, $|\alpha_1|\neq |\alpha_2|$. Then we have: $$f_{TT_n}(\alpha_1,\alpha_2)=\binom{n-1}{|\alpha_1|}.$$
\end{proposition}
\begin{proof}
Let's enumerate the vertices of $TT_n$ such that $V(T)=\lbrace v_1,v_2,\dots ,v_n\rbrace $ and $E(T)=\lbrace (v_i,v_{j});i<j \rbrace$, and suppose without loss of generality that $\alpha_1>0$ since $f_{TT_n}(\alpha)=f_{TT_n}(-\alpha)$. Consider a Hamiltonian path $P$ of $TT_n$ of type $P(\alpha_1, \alpha_2)$. Since $\alpha_1>0$ then the sink $v_n$ should be at the end of the first block. Now to construct the first block of $P$, we have to choose any $\alpha_1$ vertices from the remaining $n-1$ vertices. Once the vertices are chosen, there is only one path $P$ corresponding to this choice: in fact, since $TT_n$ is transitive, the vertices chosen to construct the first block must be aligned in an increasing order of indexes, and all remaining vertices of $TT_n$ should be aligned in a decreasing order of indexes to form the second block. Hence, the number of paths of type $P(\alpha_1,\alpha_2)$ in $TT_n$ is exactly the number of possible choices of $\alpha_1$ vertices among $n-1$, that is $\binom{n-1}{|\alpha_1|}$.
\end{proof}
\begin{remark}\label{combination2}
We have $f_{TT_n}(\alpha_1,-\alpha_1)=\dfrac{\binom{n-1}{|\alpha_1|}}{2}.$\\
In fact, following the same construction given in the proof of Proposition \ref{combination}, we may remark that if $|\alpha_1|=|\alpha_2|$, that is if $(\alpha_1,\alpha_2)$ is symmetric, then each constructed path $P$ of type $P(\alpha_1,\alpha_2)$ is in total counted twice, which clarify the necessity to divide the total number of paths of type $P(\alpha_1,\alpha_1)$ in $TT_n$ by $2$.
\end{remark}
\par As we have just noticed, a transitive tournament of order $n$ has an interesting structure, allowing us to count easily the number of some types of oriented Hamiltonian paths. In the following section, we will define a function that will allow us to compute the number of any type of oriented Hamiltonian paths in $TT_n$.\\

\section{The path-function $\mathcal{F}$}
Let $\mathcal{K}=\lbrace (\alpha_1,\alpha_2,\dots,\alpha_s)\in \mathbb{K}_s, s\in \mathbb{N}^* \rbrace$. Let $\mathcal{F}$ be the following mapping:
 \begin{eqnarray*}
  \mathcal{F} & & \mathcal{K} \longrightarrow \mathbb{N} \\
   & & (\alpha_1,\alpha_2,\dots,\alpha_s)\longrightarrow  \mathcal{F}(\alpha_1,\alpha_2,\dots,\alpha_s),
\end{eqnarray*}
\noindent defined by the recurrence relation:
\begin{eqnarray*}
  \mathcal{F}(\alpha_1,\alpha_2,\dots,\alpha_s)=& &  \mathcal{F}(\alpha_1*1,\alpha_2,\dots,\alpha_s)+ \mathcal{F}(\alpha_1,\alpha_2*1,\dots,\alpha_s) \\
  &&+\dots + \mathcal{F}(\alpha_1,\alpha_2,\dots,\alpha_s*1)
\end{eqnarray*}
where $\alpha_i*1=\alpha_i-1$ if $\alpha_i >0$ and $\alpha_i*1=\alpha_i+1$ otherwise, and satisfying: \begin{enumerate}
\item $\forall$ $t$ $\in \mathbb{N}^*$, $ \mathcal{F}(0,\alpha_2,\dots,\alpha_t)= \mathcal{F}(\alpha_2,\dots,\alpha_t)$,
\item $\forall$ $t'$ $\in \mathbb{N}^*$, $ \mathcal{F}(\alpha_1,\alpha_2,\dots,\alpha_{t'},0)= \mathcal{F}(\alpha_1,\alpha_2,\dots,\alpha_{t'})$,
\item $\forall$ $r$ $\in \mathbb{N}^*$, $ \mathcal{F}(\alpha_1,\dots,\alpha_r,0,\alpha_{r+2},\dots,\alpha_s)= \mathcal{F}(\alpha_1,\dots,\alpha_r+\alpha_{r+2},\dots,\alpha_s)$,
\item $\forall$ $\alpha$ $\in \mathbb{Z}^*$, $ \mathcal{F}(\alpha)=1$.
\end{enumerate}
We call $ \mathcal{F}$ the path-function.

\begin{theorem}\label{MAIN}
Let $TT_n$ be a transitive tournament of order $n$, and $\alpha=(\alpha_1,\alpha_2,\dots,\alpha_s)\in \mathbb{K}_s$ such that $\displaystyle\sum_{i=1}^s |\alpha_i|=n-1$. Then
$$
f_{TT_n}(\alpha) = \left\{ \begin{array}{rl} \mathcal{F}(\alpha) & \ \text{if } \alpha  \text{ is non symmetric}, \\
 & \\
\dfrac{\mathcal{F}(\alpha)}{2} & \ \text{if } \alpha  \text{ is symmetric} .
\end{array} \right.
$$
\end{theorem}
\bigskip
\smallskip
\noindent In order to prove this theorem, we need these two lemmas:

\begin{lemma}\label{L}
Let $\alpha=(\alpha_1,\alpha_2,\dots,\alpha_s)\in \mathbb{K}_s$, and $i$,$j$ such that $1\leq i < j \leq s$. We have: $$\alpha^i=-\overline{\alpha^j}\iff \alpha \ is \ symmetric \ and \ i+j=s+1.$$
\end{lemma}
\begin{proof}
We have $\alpha^i=(\alpha_1,\dots,\alpha_i*1,\dots,\alpha_j,\dots,\alpha_s)$ and $\alpha^j=(\alpha_1,\dots,\alpha_i,\dots,\alpha_j*1,\dots,\alpha_s)$ thus $-\overline{\alpha^j}=(-\alpha_s,\dots,-(\alpha_j*1),\dots,-\alpha_i,\dots,-\alpha_1)$.\\
For the sufficient condition, suppose that $\alpha$ is symmetric and $i+j=s+1$ and let's prove that $\alpha^i=-\overline{\alpha^j}$. The $i^{th}$ component of $\alpha^i$ is $\alpha_i*1$, and since $i+j=s+1$, then the $i^{th}$ component of $-\overline{\alpha^j}$ is exactly $-(\alpha_j*1)$, because the $i^{th}$ component of $\overline{\alpha^j}$ is the $(s+1-i)^{th}$ component of $\alpha^j$. Moreover, since $\alpha$ is symmetric, then $\alpha_i = -\alpha_{s+1-i} = - \alpha_j$ thus if we assume without loss of generality that $\alpha_i>0$, we have $\alpha_i*1=\alpha_i+1$, then $-(\alpha_j*1)=-(\alpha_j-1)=-\alpha_j+1$. Now, again, since $\alpha$ is symmetric, then $(\alpha_1,\alpha_2,\dots,\alpha_s)=(-\alpha_s,-\alpha_{s-1},\dots,-\alpha_1)$, and using what preceded, we get $(\alpha_1,\dots,\alpha_i+1,\dots,\alpha_j,\dots,\alpha_s)=(-\alpha_s,\dots,-\alpha_j+1,\dots,-\alpha_i,\dots,-\alpha_1)$ with $-\alpha_j+ 1$ on the $i^{th}$ position, thus $\alpha^i=-\overline{\alpha^j}$.\\
For the necessary condition, suppose that $\alpha^i=-\overline{\alpha^j}$. The $i^{th}$ component of $\alpha^i$ is $\alpha_i*1$. If $i+j\neq s+1$, then the $i^{th}$ component of $-\overline{\alpha^j}$ is equal to some $-\alpha_t$, $t\neq j$. On the other hand, the $(s+1-i)^{th}$ component of $\alpha^i$ is $\alpha_t$, and the $(s+1-i)^{th}$ component of $-\overline{\alpha^j}$ is $-\alpha_i$ since $\alpha_i$ is not modified in $\alpha^j$ because $i\neq j$. As a result, since $\alpha^i=-\overline{\alpha^j}$, we get $\alpha_i*1=-\alpha_t=-(-\alpha_i)=\alpha_i$ which is a contradiction. So $i+j=s+1$ which implies that $\alpha_j=-\alpha_i$ since the $j^{th}$ component of $\alpha^i$ is $\alpha_j$ and the $j^{th}$ component of $-\overline{\alpha^j}$ is $-\alpha_i$, and $\alpha^i=-\overline{\alpha^j}$.  Thus if we suppose w.l.o.g. that $\alpha_i>0$, then $\alpha_i*1=\alpha_i+1$, then $-(\alpha_j*1)=-(\alpha_j-1)=-\alpha_j+1$. Now since  $\alpha^i=-\overline{\alpha^j}$ then  $(\alpha_1,\dots,\alpha_i+1,\dots,\alpha_j,\dots,\alpha_s)=(-\alpha_s,\dots,-\alpha_j+1,\dots,-\alpha_i,\dots,-\alpha_1)$ with $-\alpha_j+1$ on the $i^{th}$ position (since $i+j=s+1$), and as a result $(\alpha_1,\dots,\alpha_s)=(-\alpha_s,\dots,-\alpha_1)$ so $\alpha$ is symmetric which concludes the proof.
\end{proof}
\begin{lemma}\label{LL}
Let $\alpha=(\alpha_1,\dots,\alpha_s) \in \mathbb{K}_s$, we have: $$\alpha^i \ is \  symmetric \Rightarrow \alpha^j \ is \ non-symmetric \ \forall \ 1\leq j \leq s, \ j\neq i.$$
\end{lemma}
\begin{proof}
Since $\alpha^i$ is symmetric, then $\alpha^i=-\overline{\alpha^i}$, hence $(\alpha_1,\dots,\alpha_i*1,\dots,\alpha_s)=(-\alpha_s,\dots,-(\alpha_i*1),\dots,-\alpha_1)$. Let $-\alpha_t$ be the $i^{th}$ component of $-\overline{\alpha^i}$, so $\alpha_i*1=-\alpha_t$. Suppose that there exists $j\neq i$ such that $\alpha^j$ is symmetric, then $\alpha^j=-\overline{\alpha^j}$, that is, (if we suppose without loss of generality that $i<j$), we have that $(\alpha_1,\dots,\alpha_i,\dots, \alpha_j*1,\dots,\alpha_s)=(-\alpha_s,\dots,-(\alpha_j*1),\dots,-\alpha_i,\dots,-\alpha_1)$. We have two cases to consider: If $\alpha_j$ is not on the $(s+1-i)^{th}$ position of $\alpha$, then the $i^{th}$ component of $-\overline{\alpha^j}$ is $-\alpha_t$. Thus, $\alpha_i=-\alpha_t$, a contradiction. If $\alpha_j$ is on the $(s+1-i)^{th}$ position of $\alpha$, (which means that $\alpha_i$ and $\alpha_j$ are of opposite signs because $\alpha$ has an even number of components since $\alpha^i$ is symmetric), then $-\alpha_j$ is on the $i^{th}$ position of $-\overline{\alpha^i}$ (which implies that $\alpha_i*1=-\alpha_j$ that is, if we suppose w.l.o.g. $\alpha_i>0$, $\alpha_i+1=-\alpha_j$ so $\alpha_i=-\alpha_j-1$), and $-(\alpha_j*1)$ is on the $i^{th}$ position of $-\overline{\alpha^j}$ (which implies that $\alpha_i=-(\alpha_j*1)$ i.e. $\alpha_i=-(\alpha_j-1)$ so $\alpha_i=-\alpha_j+1$), and we get a contradiction.
\end{proof}
%Remark that, since the sets of every type of Hamiltonian cycles form a partition of the set of all oriented Hamiltonian cycles in $T$)
\noindent We may now prove Theorem \ref{MAIN}:
\begin{proof}
We have:
\begin{eqnarray*}
\mathcal{F}(\alpha)= &&\mathcal{F}(\alpha_1,\alpha_2,\dots,\alpha_s) \\ = && \mathcal{F}(\alpha_1*1,\alpha_2,\dots,\alpha_s)+\mathcal{F}(\alpha_1,\alpha_2*1,\dots,\alpha_s)+\dots +\mathcal{F}(\alpha_1,\alpha_2,\dots,\alpha_s*1) \\ 
= && \mathcal{F}(\alpha^1)+\mathcal{F}(\alpha^2)+\dots +\mathcal{F}(\alpha^s).
\end{eqnarray*}
We consider 2 cases:
\begin{itemize}
\item The tuple $\alpha$ is symmetric (which implies that $s$ is even). \\
Thus obviously, all $\alpha^i$ are non-symmetric. \\
\linebreak
The proof will be done by induction on the order of the tournament $TT_n$.
The smallest case of symmetric $\alpha$ is $\alpha=(1,-1)$ which corresponds to a transitive tournament of order $3$ (acyclic triangle). In this tournament, $f_{TT_3}(1,-1)=1$ while $\mathcal{F}(1,-1)=\mathcal{F}(0,-1)+\mathcal{F}(1,0)=1+1=2$. So $f_{TT_3}(1,-1)=\frac{1}{2}\mathcal{F}(1,-1)$.\\
Suppose that the statement is true for transitive tournaments of order $l\leq n-1$, and let's prove it for $n$.\\
By Lemma \ref{L}, we have that $\alpha^i=-\overline{\alpha^j}\iff i+j=s+1$. As a result, by Proposition \ref{PP'}, if we consider the transitive tournament $\widetilde{T}=TT_n - \lbrace v \rbrace$ where $v$ is the source of $TT_n$, we have $\mathcal{P}_{\widetilde{T}}(\alpha^1)=\mathcal{P}_{\widetilde{T}}(\alpha^s)$,  $\mathcal{P}_{\widetilde{T}}(\alpha^2)=\mathcal{P}_{\widetilde{T}}(\alpha^{s-1})$, $\dots$, $\mathcal{P}_{\widetilde{T}}(\alpha^{\frac{s}{2}})=\mathcal{P}_{\widetilde{T}}(\alpha^{\frac{s}{2}+1})$, and the sets $\mathcal{P}_{\widetilde{T}}(\alpha^1)$, $\mathcal{P}_{\widetilde{T}}(\alpha^2)$, $\dots$ , $\mathcal{P}_{\widetilde{T}}(\alpha^{\frac{s}{2}})$ are pairwise different. \\
Moreover, each path of type $P(\alpha^i)$ for some $1\leq i \leq s$ is a Hamiltonian path in $\widetilde{T}$, and since all the sets of oriented Hamiltonian paths of a given type in a tournament $T$ form a partition of the set of all oriented Hamiltonian paths in $T$, then if $\mathcal{P}_{\widetilde{T}}(\alpha^i)\neq\mathcal{P}_{\widetilde{T}}(\alpha^j)$, we have $\mathcal{P}_{\widetilde{T}}(\alpha^i)\cap\mathcal{P}_{\widetilde{T}}(\alpha^j)=\emptyset$. As a result, $\mathcal{P}_{\widetilde{T}}(\alpha^1),\mathcal{P}_{\widetilde{T}}(\alpha^2),\dots,\mathcal{P}_{\widetilde{T}}(\alpha^{\frac{s}{2}})$ are all pairwise disjoint, and so are $\mathcal{P}_{\widetilde{T}}(\alpha^{\frac{s}{2}+1})$, $\mathcal{P}_{\widetilde{T}}(\alpha^{\frac{s}{2}+2})$, $\dots,\mathcal{P}_{\widetilde{T}}(\alpha^s)$.\\
\linebreak
Consider the correspondence:
\begin{eqnarray*}
  g: & & \mathcal{P}_{TT_n}(\alpha) \longrightarrow \cup_{i=1}^{\frac{s}{2}}\mathcal{P}_{\widetilde{T}}(\alpha^i) \\
   & & P \longrightarrow g(P)=P-\langle v\rangle \cup \langle \lbrace x,y \rbrace \rangle,
\end{eqnarray*}
where $x$ is the predecessor of $v$ on $P$, and $y$ its successor on $P$ if any.\\
\linebreak
Clearly, $g$ is well defined. In fact, Let $P \in \mathcal{P}_{TT_n}(\alpha_1,\alpha_2,\dots,\alpha_s)$, $P=I_1I_2\dots I_s$, and suppose that the source $v$ is the origin of some block $I_i$ of $P$ of length $\alpha_i$ ($\alpha_i>0$ since $v$ is a source), and let $x \in I_{i-1}$ be the predecessor of $v$ on $P$ and $y \in I_i$ its successor. If $(y,x)\in E(TT_n)$, then $g(P)$ is of type $P(\alpha_1,\dots,\alpha_{i-1},\alpha_i-1,\dots,\alpha_s)$, and if $(x,y)\in E(TT_n)$, then $g(P)$ is of type $P(\alpha_1,\dots,\alpha_{i-1}+1,\alpha_i,\dots,\alpha_s)$, and both of them belong to $\cup_{i=1}^{\frac{s}{2}}\mathcal{P}_{\widetilde{T}}(\alpha^i)$, (if $g(P) \in \mathcal{P}_{\widetilde{T}}(\alpha^i)$ for some $\frac{s}{2}+1 \leq i \leq s$, then as previously mentioned, it belongs to a set $\mathcal{P}_{\widetilde{T}}(\alpha^i)$ for some $1 \leq i \leq \frac{s}{2}$ and so it belongs to $\cup_{i=1}^{\frac{s}{2}}\mathcal{P}_{\widetilde{T}}(\alpha^i)$). Moreover, it is obvious that $g$ is a mapping.\\
\linebreak
\noindent The mapping $g$ is a bijection: \\
It is surjective: Let $P'\in \cup_{i=1}^{\frac{s}{2}}\mathcal{P}_{\widetilde{T}}(\alpha^i)$, then $\exists$ $1\leq i \leq \frac{s}{2}$ such that $P' \in \mathcal{P}_{\widetilde{T}}(\alpha^i)$. Suppose that $\alpha_i>0$, then $P' \in \mathcal{P}_{\widetilde{T}}(\alpha_1,\dots,\alpha_{i-1},\alpha_i-1,\alpha_{i+1},\dots,\alpha_s)$ and let $x$ be the origin of the block $I_i$ of length $\alpha_i-1$, and $y \in I_{i-1}$ its predecessor, and write $P'=P_1\cup (x,y)\cup P_2$. Since $v$ is a source, then $(v,y)$ and $(v,x)$ $\in E(TT_n)$, hence $P=P_1\cup (v,y)\cup (v,x)\cup P_2$ is of type $P(\alpha_1,\alpha_2,\dots,\alpha_s)$, so it belongs to  $\mathcal{P}_{TT_n}(\alpha)$ with $g(P)=P'$. The case $\alpha_i<0$ is similar.\\
Also, $g$ is injective: Let $P$ and $P'$ be two paths in $TT_n$ of type $P(\alpha)$, such that $g(P)=g(P')$ $\in \mathcal{P}_{\widetilde{T}}(\alpha^i)=\mathcal{P}_{\widetilde{T}}(\alpha_1,\dots,\alpha_{i-1},\alpha_i*1,\alpha_{i+1},\dots,\alpha_s)$ for some $1\leq i \leq \frac{s}{2}$. Suppose that $\alpha_i>0$. %Write $g(P)=I_1I_2\dots I_s$ and $g(P')=I'_1I'_2\dots I'_s$. Suppose that $\alpha_i>0$. %Since $g(P) \in \mathcal{P}_{\widetilde{T}}(\alpha^i)= \mathcal{P}_{\widetilde{T}}(\alpha_1,\dots,\alpha_{i-1},\alpha_i-1,\alpha_{i+1},\dots,\alpha_s)$ where $\alpha^i$ is non-symmetric, so there is only one enumeration of the vertices of $g(P)$ such that $g(P)$ is of type $P(\alpha^i)$ with respect to this enumeration.
So let $g(P)=I_1I_2\dots I_s=u_1u_2...u_ryx w_1w_2...w_t$ where $x$ is the origin of the block $I_i$ of length $\alpha_i-1$, and $y \in I_{i-1}$ its predecessor, and the arc $(x,y)$ replaced the arcs $(v,x)$ and $(v,y)$ in $P=u_1u_2...u_ryvx w_1w_2...w_t$. The path $g(P)=u_1u_2...u_ryx w_1w_2...w_t$ is of type $P(\alpha^i)$ with respect to this enumeration. %Similarly, since $g(P')\in \mathcal{P}_{\widetilde{T}}(\alpha^i)= \mathcal{P}_{\widetilde{T}}(\alpha_1,\dots,\alpha_{i-1},\alpha_i-1,\alpha_{i+1},\dots,\alpha_s)$ where $\alpha^i$ non-symmetric,
Also, let $g(P')=I'_1I'_2\dots I'_s$ $=u'_1u'_2...u'_ry'x' w'_1w'_2...w'_t$ where $x'$ is the origin of the block $I'_i$ of length $\alpha_i-1$, and $y \in I'_{i-1}$ its predecessor, and the arc $(x',y')$ replaced the arcs $(v,x')$ and $(v,y')$ in $P'=u'_1u'_2...u'_ry'vx' w'_1w'_2...w'_t$. The path $g(P')=u'_1u'_2...u'_ry'x' w'_1w'_2...w'_t$ is of type $P(\alpha^i)$ with respect to this enumeration. %Now, since $(\alpha^i)$ is not symmetric, then $g(P')$ can't be of type $P(\alpha^i)$ with respect to the enumeration $w'_tw'_{t-1}...w'_1x'y'u'_ru'_{r-1}...u'_1$ thus
Now, since $g(P)=g(P')$ then we either have $(u_1,u_2,...,u_r,y,x, w_1,w_2,...,w_t)=(u'_1,u'_2,...,u'_r,y',x', w'_1,w'_2,...,w'_t)$ or we have $(u_1,u_2,...,u_r,y,x, w_1,w_2,...,w_t)=$ \\ $(w'_t,w'_{t-1},...,w'_1,x',y', u'_r,u'_{r-1},...,u'_1)$. If the second case is true, then the path $g(P')=w'_tw'_{t-1}...w'_1x'y' u'_ru'_{r-1}...u'_1$ is of type $P(\alpha^i)$ with respect to this enumeration, which is impossible since $\alpha^i$ is non-symmetric. Thus, only the first case is true, and adding the arcs $(v,y)=(v,y')$ and $(v,x)=(v,x')$ we get $P=u_1u_2...u_ryvx w_1w_2...w_t=P'=u'_1u'_2...u'_ry'vx' w'_1w'_2...w'_t$. The case $\alpha_i<0$ is similar.\\
\linebreak
Now, since $g$ is a bijection, then $f_{TT_n}(\alpha)=\sum_{i=1}^{\frac{s}{2}}f_{\widetilde{T}}(\alpha^i)$. Since by induction we have $ f_{\widetilde{T}}(\alpha^i)=\mathcal{F}(\alpha^i)$ because the order of $\widetilde{T}$ is $n-1$ and all $\alpha^i$ are non-symmetric, then $$f_{TT_n}(\alpha)=\sum_{i=1}^{\frac{s}{2}}f_{\widetilde{T}}(\alpha^i)=\sum_{i=1}^{\frac{s}{2}}\mathcal{F}(\alpha^i)=\frac{1}{2}\sum_{i=1}^s \mathcal{F}(\alpha^i)=\frac{1}{2}\mathcal{F}(\alpha).$$
\smallskip
\item The tuple $\alpha$ is non-symmetric. \\
\linebreak
The proof will also be done by induction on the order of the tournament $TT_n$.
The smallest case of non-symmetric $\alpha$ is $\alpha=(2)$ (i.e. directed Hamiltonian paths) which corresponds also to the transitive tournament of order $3$ (acyclic triangle). In this tournament, $f_{TT_3}(2)=1$ and $\mathcal{F}(2)=1$ by the definition of the mapping $\mathcal{F}$. So $f_{TT_3}(2)=\mathcal{F}(2)$.\\
Suppose that the statement is true for transitive tournaments of order $l\leq n-1$, and let's prove it for $n$.\\
Since $\alpha$ is non-symmetric, we have by Lemma \ref{L} that $\alpha^i\neq -\overline{\alpha^j}$ $\forall$ $1\leq i,j\leq s$. As a result, if we consider the transitive tournament $\widetilde{T}=TT_n - \lbrace v \rbrace$ where $v$ is the source of $TT_n$, then also by Proposition \ref{PP'}, $\mathcal{P}_{\widetilde{T}}(\alpha^i)\neq \mathcal{P}_{\widetilde{T}}(\alpha^j)$ $\forall$ $1\leq i,j\leq s$, so we have $\mathcal{P}_{\widetilde{T}}(\alpha^i)\cap\mathcal{P}_{\widetilde{T}}(\alpha^j)=\emptyset$ $\forall$ $1\leq i,j\leq s$, hence $\mathcal{P}_{\widetilde{T}}(\alpha^1),\mathcal{P}_{\widetilde{T}}(\alpha^2),\dots,\mathcal{P}_{\widetilde{T}}(\alpha^s)$ are all pairwise disjoint. \\
\linebreak
Consider the correspondence:
\begin{eqnarray*}
  g': & & \mathcal{P}_{TT_n}(\alpha) \longrightarrow \cup_{i=1}^{s}\mathcal{P}_{\widetilde{T}}(\alpha^i) \\
   & & P \longrightarrow g'(P)=P-\langle v\rangle \cup \langle \lbrace x,y \rbrace \rangle,
\end{eqnarray*}
where $v$ is the source of $TT_n$, $x$ is the predecessor of $v$ on $P$, and $y$ its successor on $P$ if any.\\
\linebreak
As in the previous case, we can prove that $g'$ is a surjective mapping. However, $g'$ is not always injective. In fact, let $P'\in \cup_{i=1}^{s}\mathcal{P}_{\widetilde{T}}(\alpha^i)$, and let's find how many $P\in \mathcal{P}_{TT_n}(\alpha)$ there exist, such that $g'(P)=P'$. We consider two cases:\\
\begin{enumerate}
\item All $\alpha^i$ are non-symmetric. \\
Then following the same arguments as given in the first case to prove that $g'$ is injective, we may prove that $P'$ has only one antecedent in $\mathcal{P}_{TT_n}(\alpha)$ thus $g'$ is injective. So $g'$ is a bijection.\\
Now, since $g'$ is a bijection, then $f_{TT_n}(\alpha)=\sum_{i=1}^{s}f_{\widetilde{T}}(\alpha^i)$. And since by induction we have $ f_{\widetilde{T}}(\alpha^i)=\mathcal{F}(\alpha^i)$ (the order of $\widetilde{T}$ is $n-1$ and all $\alpha^i$ are non-symmetric), then $$f_{TT_n}(\alpha)=\sum_{i=1}^{s}f_{\widetilde{T}}(\alpha^i)=\sum_{i=1}^{s}\mathcal{F}(\alpha^i)=\mathcal{F}(\alpha).$$
\smallskip
\item There exists $1\leq i_0 \leq s$ such that $\alpha^{i_0}$ is symmetric.\\
Then by Lemma \ref{LL}, all $\alpha^j$, $1\leq j\neq i_0 \leq s$, are non-symmetric.\\
\linebreak
Since $P'\in \cup_{i=1}^{s}\mathcal{P}_{\widetilde{T}}(\alpha^i)$, then $\exists$ $1\leq i \leq s$ such that $P' \in \mathcal{P}_{\widetilde{T}}(\alpha^i)$. Suppose $\alpha_i>0$, then $P' \in \mathcal{P}_{\widetilde{T}}(\alpha_1,\dots,\alpha_{i-1},\alpha_i-1,\alpha_{i+1},\dots,\alpha_s)$. (The case $\alpha_i<0$ is similar).\\
Write $P'=I_1I_2\dots I_s=u_1u_2...u_ryx w_1w_2...w_t$ where $x$ is the origin of the block $I_i$ of length $\alpha_i-1$, and $y \in I_{i-1}$ its predecessor.\\
If $i=i_0$, then $\alpha^i$ is symmetric, then $P'$ is also of type $P(\alpha^i)$ with respect to the other enumeration: $w_tw_{t-1}...w_1xyu_ru_{r-1}...u_1$ and we rewrite $P'$ with respect to this enumeration as $P'=I'_1I'_2\dots I'_s=z_1z_2...z_ry'x'v_1v_2...v_t$ where $x'$ is the origin of the block $I'_i$ of length $\alpha_i-1$, and $y' \in I'_{i-1}$ its predecessor. So now, we consider the two paths $P_1=u_1u_2...u_ryvx w_1w_2...w_t$ where the arc $(x,y)$ is replaced by the arcs $(v,x)$ and $(v,y)$, and $P_2=z_1z_2...z_ry'vx'v_1v_2...v_t$ where the arc $(x',y')$ is replaced by the arcs $(v,x')$ and $(v,y')$. They are distinct and both of type $P(\alpha)$, and they are the only ones such that $g'(P_1)=g'(P_2)=P'$.\\
If $i\neq i_0$, then $\alpha^i$ is non-symmetric, thus $P'$ can't be of type $P(\alpha^i)$ with respect to the other enumeration: $w_tw_{t-1}...w_1xyu_ru_{r-1}...u_1$. So there is only one path $P=u_1u_2...u_ryvx w_1w_2...w_t$ where the arc $(x,y)$ in $P'$ is replaced by the arcs $(v,x)$ and $(v,y)$, that is of type $P(\alpha)$, and such that $g'(P)=P'$.\\
\linebreak
A a result, we have: $f_{TT_n}(\alpha)=2.f_{\widetilde{T}}(\alpha^{i_0})+\sum_{i\neq i_0,i=1}^s f_{\widetilde{T}}(\alpha^i)$. Now, by induction, $f_{\widetilde{T}}(\alpha^{i_0})=\frac{1}{2}\mathcal{F}(\alpha^{i_0})$, because $\alpha^{i_0}$ is symmetric, and $\forall$ $i\neq i_0$, $f_{\widetilde{T}}(\alpha^{i})=\mathcal{F}(\alpha^{i})$, because $\alpha^i$ is non-symmetric, thus $$f_{TT_n}(\alpha)=\sum_{i=1}^s \mathcal{F}(\alpha^i)=\mathcal{F}(\alpha),$$
and this concludes the proof.
\end{enumerate}
\end{itemize}
\end{proof}

\section{Some properties of $\mathcal{F}$}

In this section, we will study the path-function $\mathcal{F}$. As we previously mentioned, in a transitive tournament $TT_n$ we have $f_{TT_n}(\alpha)=f_{TT_n}(-\alpha)$ $\forall$ $\alpha \in \mathbb{K}_s$. For that reason, we will redefine the mapping $f$ by removing the signs of the components of the tuples.\\
\linebreak
\noindent Let $\mathcal{N}=\lbrace (\alpha_1,\dots,\alpha_s)\in(\mathbb{N}^*)^s,s\in \mathbb{N}^*\rbrace$, $\mathcal{F}$ is defined by:
\begin{eqnarray*}
  \mathcal{F}: & & \mathcal{N} \longrightarrow \mathbb{N} \\
   & & (\alpha_1,\alpha_2,\dots,\alpha_s)\longrightarrow \mathcal{F}(\alpha_1,\alpha_2,\dots,\alpha_s),
\end{eqnarray*}
\noindent under the recurrence relation:
\begin{eqnarray*}
  \mathcal{F}(\alpha_1,\alpha_2,\dots,\alpha_s)=& & \mathcal{F}(\alpha_1-1,\alpha_2,\dots,\alpha_s)+\mathcal{F}(\alpha_1,\alpha_2-1,\dots,\alpha_s) \\
  &&+\dots +\mathcal{F}(\alpha_1,\alpha_2,\dots,\alpha_s-1)
\end{eqnarray*}
\noindent satisfying the properties: \begin{enumerate}
\item $\forall$ $t$ $\in \mathbb{N}^*$, $\mathcal{F}(0,\alpha_2,\dots,\alpha_t)=\mathcal{F}(\alpha_2,\dots,\alpha_t)$,
\item $\forall$ $t'$ $\in \mathbb{N}^*$, $\mathcal{F}(\alpha_1,\alpha_2,\dots,\alpha_{t'},0)=\mathcal{F}(\alpha_1,\alpha_2,\dots,\alpha_{t'})$,
\item $\forall$ $r$ $\in \mathbb{N}^*$, $\mathcal{F}(\alpha_1,\dots,\alpha_r,0,\alpha_{r+2},\dots,\alpha_s)=\mathcal{F}(\alpha_1,\dots,\alpha_r+\alpha_{r+2},\dots,\alpha_s)$,
\item $\forall$ $\alpha$ $\in \mathbb{N}^*$, $\mathcal{F}(\alpha)=1$.
\end{enumerate}
\smallskip
\begin{remark}
With this definition of $\mathcal{F}$, $\forall$ $\alpha=(\alpha_1,\dots,\alpha_s) \in \mathbb{K}_s$, the number computed now by $\mathcal{F}$ is $\mathcal{F}(|\alpha_1|,|\alpha_2|,\dots,|\alpha_s|)$, which will be either $f_{TT_n}(\alpha)$ or $f_{TT_n}(-\alpha)$ if $\alpha$ is non symmetric (resp. either $\frac{f_{TT_n}(\alpha)}{2}$ or $\frac{f_{TT_n}(-\alpha)}{2}$ if $\alpha$ is symmetric). \\
Note that this does not exclude the fact that the sets $\mathcal{P}_{TT_n}(\alpha)$ and $\mathcal{P}_{TT_n}(-\alpha)$ can be either disjoint or the same.
\end{remark}

\noindent Obviously, $\mathcal{F}(a_1,a_2,\dots,a_s)=\mathcal{F}(a_s,a_{s-1},\dots,a_1)$, $\forall \ \alpha=(\alpha_1,\dots,\alpha_s)\in \mathcal{N}$, since $\mathcal{F}$ is a symmetric function. We also have the following properties:
\begin{remark}\label{1,s block}
For all $\alpha=(\alpha_1,\alpha_2,\dots,\alpha_s)\in \mathcal{N}$, $s\geq 2$, and for all $m\in \mathbb{N}^*$ we have: $$\mathcal{F}(m) < \mathcal{F}(\alpha_1,\alpha_2,\dots,\alpha_s),$$
\end{remark}
\noindent Since by the definition of $\mathcal{F}$ we have $\mathcal{F}(m)=1$ and $\mathcal{F}(\alpha_1,\dots,\alpha_s)>1$ for $s>1$.
\begin{proposition}\label{2 block}
For all $m,n,m',n'\in \mathbb{N}^*$, $m+n=m'+n'$, we have:
\begin{itemize}
    \item $\mathcal{F}(m,n)=\binom{m+n}{m}$,
    \item $\mathcal{F}(m,n)<\mathcal{F}(m',n')\iff mn<m'n'$.
\end{itemize}
\end{proposition}
\begin{proof}
The first statement is a direct result of Proposition \ref{combination}, Remark \ref{combination2} and Theorem \ref{MAIN}. For the second statement, we know that for $ m,n,m',n' \in \mathbb{N}^*$ such that $m+n=m'+n'=p$ and $mn<m'n'$ we have $\binom{p}{m}=\binom{p}{n}< \binom{p}{m'}=\binom{p}{n'}$. This is equivalent to $\mathcal{F}(m,n)<\mathcal{F}(m',n')$, since $\mathcal{F}(m,n)=\binom{m+n}{m}$.
\end{proof}
\smallskip
\begin{proposition}\label{2,3 block}
For all $m,n,t,a\in \mathbb{N}^*$, $a+m+n=a+t$, we have: $$\mathcal{F}(a,t) < \mathcal{F}(a,m,n).$$
\end{proposition}
\begin{proof}
The proof will be done by induction on $p=a+t=a+m+n$, that is by induction on the order $p+1$ of a transitive tournament. \\
The smallest tournament in which we can compare $\mathcal{F}(a,t)$ and $\mathcal{F}(a,m,n)$ has $4$ vertices, that is $p=3$, with $a=1$, $t=2$, and $m=n=1$.\\
We have $\mathcal{F}(1,2)=\mathcal{F}(2)+\mathcal{F}(1,1)$, while $\mathcal{F}(1,1,1)=\mathcal{F}(1,1)+\mathcal{F}(2)+\mathcal{F}(1,1)$ and the inequality follows.\\
Suppose that the inequality is true for a tournament of order less than or equal to $p$, and let $T$ be a tournament of order $p+1$, with $a+t=a+m+n=p$.\\
Note that we necessarily have $m$ and $n <t$.\\
We have $\mathcal{F}(a,t)=\mathcal{F}(a-1,t)+\mathcal{F}(a,t-1)$, and $\mathcal{F}(a,m,n)=\mathcal{F}(a-1,m,n)+\mathcal{F}(a,m-1,n)+\mathcal{F}(a,m,n-1)$.\\
First, we have $\mathcal{F}(a-1,t)<\mathcal{F}(a-1,m,n)$. In fact, if $a-1\neq 0$, it is true by induction, while if $a-1=0$, it follows from Remark \ref{1,s block} that $\mathcal{F}(t)<\mathcal{F}(m,n)$.\\
Now, if $m>1$ (resp. $n>1$), we have by induction that $\mathcal{F}(a,t-1) < \mathcal{F}(a,m-1,n)$ (resp. $\mathcal{F}(a,t-1) < \mathcal{F}(a,m,n-1))$, and the inequality follows.\\
If $m=n=1$ (which implies $t=2$), then $\mathcal{F}(a,t-1)=\mathcal{F}(a,1) <\mathcal{F}(a,m-1,n)+\mathcal{F}(a,m,n-1)=\mathcal{F}(a+1)+\mathcal{F}(a,1)$, and we get the desired inequality.
\end{proof}
\smallskip
\begin{proposition}\label{3 block}
For all $m,n,m',n',a\in \mathbb{N}^*$, $a+m+n=a+m'+n'$, we have:
\begin{itemize}
\item $\mathcal{F}(a,m,n)<\mathcal{F}(a,m',n')\iff (mn<m'n') \ or \ (mn=m'n' , \ m<n)$.
\item $\mathcal{F}(m,a,n)<\mathcal{F}(m',a,n')\iff (mn<m'n')$.
\end{itemize}
\end{proposition}
\begin{proof}
We begin by the first statement.\\
\linebreak
Suppose first that $mn=m'n'$ with $m<n$. \\
Since $m+n=m'+n'$ then $m=n'$ and $n=m'$. So we need to compare $\mathcal{F}(a,m,n)$ and $\mathcal{F}(a,n,m)$. The proof will be done by induction on $p=m+n$ (so by induction on the order $p+1$ of the tournament).\\
If $p=4$ (tournament on 5 vertices) then the initial step is to compare $\mathcal{F}(1,1,2)$ and $\mathcal{F}(1,2,1)$. $\mathcal{F}(1,1,2)=\mathcal{F}(1,2)+\mathcal{F}(3)+\mathcal{F}(1,1,1)$ and $\mathcal{F}(1,2,1)=\mathcal{F}(2,1)+\mathcal{F}(1,1,1)+\mathcal{F}(1,2)$. Since $\mathcal{F}(3)<\mathcal{F}(2,1)$ by Remark \ref{1,s block}, the result follows.\\
Suppose it's true till $p-1$, i.e. for a tournament of order less than or equal $p$. So let $T$ be a tournament of order $p+1$, $a+m+n=p$, $m<n$.\\
We have $\mathcal{F}(a,m,n)=\mathcal{F}(a,m-1,n)+\mathcal{F}(a,m,n-1)+\mathcal{F}(a-1,m,n)$ and $\mathcal{F}(a,n,m)=\mathcal{F}(a,n,m-1)+\mathcal{F}(a,n-1,m)+\mathcal{F}(a-1,n,m)$.\\
Now by induction, $m-1<n$ so $\mathcal{F}(a,m-1,n)<\mathcal{F}(a,n,m-1)$ (note that if $m=1$, then $\mathcal{F}(a,m-1,n)=\mathcal{F}(a+n)=1< \mathcal{F}(a,n,m-1)=\mathcal{F}(a,n)$ by Remark \ref{1,s block}), and $m\leq n-1$ so $\mathcal{F}(a,m,n-1)\leq \mathcal{F}(a,n-1,m)$, and if $a-1\neq 0$, $\mathcal{F}(a-1,m,n)<\mathcal{F}(a-1,n,m)$, while if $a-1=0$ then $\mathcal{F}(a-1,m,n)=\mathcal{F}(m,n)=\mathcal{F}(n,m)= \mathcal{F}(a-1,n,m)$. Hence $\mathcal{F}(a,m,n)<\mathcal{F}(a,n,m)$.\\
\linebreak
Suppose now that $mn<m'n'$. \\
Since by the first case, $\mathcal{F}(a,n,m)<\mathcal{F}(a,m,n)$ if $n<m$ and $\mathcal{F}(a,m',n')<\mathcal{F}(a,n',m')$ if $m'<n'$, let us suppose that $n\leq m$ and $m'\leq n'$, and prove that $\mathcal{F}(a,m,n)<\mathcal{F}(a,m',n')$.\\
Remark that $mn<m'n'$ and $m+n=m'+n'$ both imply that $m-n\geq 2$, because we can't find $m'$ and $n'$ such that $m+n=m'+n'$ and $mn<m'n'$ when $m=n$ or $m=n+1$ because $m.n$ is maximal. We also have $m',n' > n$ and $m',n'<m$.\\
Let $p=a+m+n=a+m'+n'$. We will also do the proof by induction on $p$.\\
The smallest tournament to satisfy $mn<m'n'$ has 6 vertices (i.e. $p=5$). That is for $m=3$, $n=1$, $m'=2$, $n'=2$, $a=1$.\\
We have $\mathcal{F}(1,3,1)=\mathcal{F}(3,1)+\mathcal{F}(1,2,1)+\mathcal{F}(1,3)$ and $\mathcal{F}(1,2,2)=\mathcal{F}(2,2)+\mathcal{F}(1,1,2)+\mathcal{F}(1,2,1)$. By Proposition \ref{2 block} we have $\mathcal{F}(3,1)<\mathcal{F}(2,2)$ and by a simple calculation we have $\mathcal{F}(1,3)<\mathcal{F}(1,1,2)$, and then we get $\mathcal{F}(1,3,1)<\mathcal{F}(1,2,2)$.\\
Suppose that the statement is true till $p-1$, that is for a tournament of order less than or equal $p$. Let $T$ be a tournament of order $p+1$, $a+m+n=a+m'+n'=p$, with $mn<m'n'$. We have:\\
$\mathcal{F}(a,m,n)=\mathcal{F}(a-1,m,n)+\mathcal{F}(a,m-1,n)+\mathcal{F}(a,m,n-1)$ and $\mathcal{F}(a,m',n')=\mathcal{F}(a-1,m',n')+\mathcal{F}(a,m'-1,n')+ \mathcal{F}(a,m',n'-1)$.\\
We have $\mathcal{F}(a-1,m,n)<\mathcal{F}(a-1,m',n')$ (by induction if $a>1$ and by Proposition \ref{2 block} if $a=1$).
Now, $m>n$ implies that $$(m-1)n=\dfrac{(p-a)^2-(m-n-1)^2}{4} \quad \text{and} \quad m(n-1)=\dfrac{(p-a)^2-(m-n+1)^2}{4}.$$

\noindent We will consider two cases:
\begin{itemize}
\item Case 1: $m'<n'$. Hence \\ $$(m'-1)n'=\dfrac{(p-a)^2-(n'-m'+1)^2}{4},$$and$$m'(n'-1)=\dfrac{(p-a)^2-(n'-m'-1)^2}{4}.$$\\
As a result, $(m-1)n<m'(n'-1)$ and $m(n-1)<(m'-1)n'$, so by induction we have $\mathcal{F}(a,m-1,n)<\mathcal{F}(a,m',n'-1)$ (since $m-1\geq n$ and $m'\leq n'-1$) and $\mathcal{F}(a,m,n-1)<\mathcal{F}(a,m'-1,n')$ (by induction if $n-1\neq 0$ and by Proposition \ref{2,3 block} if $n-1=0$). So we finally get $\mathcal{F}(a,m,n)<\mathcal{F}(a,m',n')$.\\
\item Case 2: $m'=n'$. Hence $$(m'-1)n'=m'(n'-1)=\dfrac{(p-a)^2-1}{4}.$$\\
If $m-n>2$ then $m-n-1>1$, so $(m-1)n<m'(n'-1)$ thus $\mathcal{F}(a,m-1,n)<\mathcal{F}(a,n'-1,m')$ (by induction, since $m-1\geq n$ and $n'-1\leq m'$) $<\mathcal{F}(a,m',n'-1)$ (by the case treated above). And since $m(n-1)<(m'-1)n'$ then $\mathcal{F}(a,m,n-1)<\mathcal{F}(a,m'-1,n')$ (by induction if $n-1\neq 0$ and by Proposition \ref{2,3 block} if $n-1=0$). We finally get $\mathcal{F}(a,m,n)<\mathcal{F}(a,m'n')$.\\
If $m-n=2$ then $m=n+2$, ans since $m>n'=m'>n$ (bc $m+n=n'+m'$ and $mn<m'n'$) then $n'=m'=n+1$. Now we have: $\mathcal{F}(a,m,n)=\mathcal{F}(a,n+2,n)=\mathcal{F}(a-1,n+2,n)+\mathcal{F}(a,n+1,n)+\mathcal{F}(a,n+2,n-1)$ and $\mathcal{F}(a,m',n')=\mathcal{F}(a,n+1,n+1)=\mathcal{F}(a-1,n+1,n+1)+\mathcal{F}(a,n,n+1)+\mathcal{F}(a,n+1,n)$. Now $(n+2)(n-1)=n^2+n-2$ and $n(n+1)=n^2+n$ so $(n+2)(n-1)<n(n+1)$, so $\mathcal{F}(a,n+2,n-1)<\mathcal{F}(a,n,n+1)$ (by induction if $n-1\neq 0$ and by Proposition \ref{2,3 block} if $n-1=0$). Finally we get $\mathcal{F}(a,m,n)<\mathcal{F}(a,m'n')$.
\end{itemize}
\medskip
For the necessary condition, suppose that $\mathcal{F}(a,m,n)<\mathcal{F}(a,m',n')$ with $mn\geq m'n'$ and let's prove $mn=m'n'$ with $m<n$. If $mn>m'n'$ then $\mathcal{F}(a,m,n)>\mathcal{F}(a,m',n')$ by the sufficient condition, which is a contradiction. If $mn=m'n'$ and since $m+n=m'+n'$ we have many cases: If  $m=m', n=n'$ then $\mathcal{F}(a,m,n)=\mathcal{F}(a,m',n')$, a contradiction. If $m=n', n=m'$ then if $m>n$, $\mathcal{F}(a,m,n)>\mathcal{F}(a,m',n')$ by the sufficient condition, a contradiction too. Thus $m<n$.\\
\linebreak
Now we prove the second statement: \\
For the sufficient condition, the proof is similar to the proof of the previous one for $mn<m'n'$. Only note that if $a=1$, we will have in the induction that $\mathcal{F}(m,a-1,n)=\mathcal{F}(m+n)$ and $\mathcal{F}(m',a-1,n')=\mathcal{F}(m'+n')$, so $\mathcal{F}(m,a-1,n)=\mathcal{F}(m',a-1,n')=1$. But this won't cause a problem in proving $\mathcal{F}(m,a,n)<\mathcal{F}(m',a,n')$ for $mn<m'n'$ because the other terms will lead strict inequalities. For the necessary condition, suppose that $\mathcal{F}(m,a,n)<\mathcal{F}(m',a,n')$ and let's prove that $mn<m'n'$. If $mn>m'n'$ then $\mathcal{F}(m,a,n)>\mathcal{F}(m',a,n')$ by the sufficient condition, a contradiction. If $mn=m'n'$ and since $m+n=m'+n'$, then either $m=m', n=n'$ or $m=n', n=m'$. But in both cases we get $\mathcal{F}(m,a,n)=\mathcal{F}(m',a,n')$, a contradiction. Thus $mn<m'n'$.
\end{proof}
\smallskip
\begin{proposition}
For all $m,n,a,b\in \mathbb{N}^*$ we have:
$$m<n \text{ and } a<b \ \Rightarrow \mathcal{F}(m,a,b,n)>\mathcal{F}(m,b,a,n).$$
\end{proposition}
\begin{proof}
The proof is also done by induction on $p=m+a+b+n$. The initial step is to compare $\mathcal{F}(1,2,1,2)$ and $\mathcal{F}(1,1,2,2)$, i.e. for $p=6$, $m=1$, $n=2$, $a=1$, $b=2$. We have $\mathcal{F}(1,2,1,2)=\mathcal{F}(2,1,2)+\mathcal{F}(1,1,1,2)+\mathcal{F}(1,4)+\mathcal{F}(1,2,1,1)$ and $\mathcal{F}(1,1,2,2)=\mathcal{F}(1,2,2)+\mathcal{F}(3,2)+\mathcal{F}(1,1,1,2)+\mathcal{F}(1,1,2,1)$. Now by Proposition \ref{2 block} we have $\mathcal{F}(1,4)<\mathcal{F}(3,2)$ and by Proposition \ref{3 block} we have $\mathcal{F}(2,1,2)<\mathcal{F}(1,2,2)$, moreover $\mathcal{F}(1,1,2,1)=\mathcal{F}(1,2,1,1)$, and we get $\mathcal{F}(1,2,1,2)<\mathcal{F}(1,1,2,2)$. \\
Suppose that the statement is true till $p-1$ (tournament having at most $p$ vertices). So let $T$ be a tournament on $p+1$ vertices, $m+n+a+b=p$, and suppose $m<n$, $a<b$.\\
We have $\mathcal{F}(m,a,b,n)=\mathcal{F}(m-1,a,b,n)+\mathcal{F}(m,a-1,b,n)+\mathcal{F}(m,a,b-1,n)+\mathcal{F}(m,a,b,n-1)$ and $\mathcal{F}(m,b,a,n)=\mathcal{F}(m-1,b,a,n)+\mathcal{F}(m,b-1,a,n)+\mathcal{F}(m,b,a-1,n)+\mathcal{F}(m,b,a,n-1)$. We have the following inequalities:
\begin{itemize}
\item $\mathcal{F}(m,a-1,b,n)>\mathcal{F}(m,b,a-1,n)$. In fact, $a-1<b$, so: If $a-1\neq 0$, by induction we have $\mathcal{F}(m,a-1,b,n)>\mathcal{F}(m,b,a-1,n)$. If $a-1=0$ then $\mathcal{F}(m,a-1,b,n)=\mathcal{F}(m+b,n)$ and $\mathcal{F}(m,b,a-1,n)=\mathcal{F}(m,b+n)$. However, $(m+b)+n=m+(b+n)$ and $(m+b)n=mn+bn>mn+bm=m(b+n)$, so by Proposition \ref{2 block} we have $\mathcal{F}(m+b,n)>\mathcal{F}(m,b+n)$.
\item $\mathcal{F}(m-1,a,b,n)>\mathcal{F}(m-1,b,a,n)$. In fact, if $m-1>0$ then it's true by induction. If $m-1=0$ then since $b>a$, we have $\mathcal{F}(a,b,n)>\mathcal{F}(b,a,n)$ by Proposition \ref{3 block}.
\item $\mathcal{F}(m,a,b-1,n)\geq \mathcal{F}(m,b-1,a,n)$. In fact, if $a<b-1$ then by induction it's true. If $a=b-1$, then $\mathcal{F}(m,a,b-1,n)=\mathcal{F}(m,a,a,n)=\mathcal{F}(m,b-1,a,n)$.
\item $\mathcal{F}(m,a,b,n-1)\geq \mathcal{F}(m,b,a,n-1)$. In fact, if $m<n-1$ it's true by induction. If $m=n-1$ then $\mathcal{F}(m,a,b,n-1)=\mathcal{F}(m,a,b,m)=\mathcal{F}(m,b,a,m)=\mathcal{F}(m,b,a,n-1)$
\end{itemize}
Then using all these inequalities, we get the result.
\end{proof}
\medskip
\noindent On the other hand, using a program for computing $\mathcal{F}$ (which we will discuss in the next section) , we may prove the following statements wrong, even if they seem to be true:
\smallskip
\begin{itemize}
\item For $\alpha=(\alpha_1,\dots,\alpha_s)\in (\mathbb{N}^*)^s$ and $\beta=(\beta_1,\dots,\beta_t) \in (\mathbb{N}^*)^t$, $2\leq s <t$, $\sum_{i=1}^{s} \alpha_i=\sum_{i=1}^t \beta_i$, then $\mathcal{F}(\alpha)<\mathcal{F}(\beta)$.\\
\underline{Counter-example:} $\mathcal{F}(3,3)=20=\mathcal{F}(1,1,4)$, $\mathcal{F}(3,4)=35>\mathcal{F}(1,1,5)=27$.
\item For $a,m,n,m',n' \in \mathbb{N}^*$ with $a+m+n=a+m'+n'$ then $mn<m'n'\Rightarrow \mathcal{F}(a,m,n)<\mathcal{F}(m',a,n')$.\\
\underline{Counter-example:} $\mathcal{F}(1,2,4)=85>\mathcal{F}(3,1,3)=69$ while $2.4=8 < 3.3=9$.
\item For $a,m,n,m',n' \in \mathbb{N}^*$ with $a+m+n=a+m'+n'$ then $mn<m'n'\Rightarrow \mathcal{F}(m,a,n)<\mathcal{F}(a,m',n')$.\\
\underline{Counter-examples:} $\mathcal{F}(2,11,5)=637924>\mathcal{F}(11,3,4)=631787$, $\mathcal{F}(2,12,5)=1015988>\mathcal{F}(12,3,4)=984503$, while $2.5=10<3,4=12)$.
%However, for the same $m,n,m',n'$ but for $a<11$, the statement is true, for example: $\mathcal{F}(2,7,5)=65857<\mathcal{F}(7,3,4)=74502$, $\mathcal{F}(2,8,5)=125411<\mathcal{F}(8,3,4)=135927$, $\mathcal{F}(2,9,5)=225589<\mathcal{F}(9,3,4)=236027$, $\mathcal{F}(2,10,5)=387023<\mathcal{F}(10,3,4)=393107$.\\
%So we remark that the first term is smaller than the second term, but the difference between them decreases as $a$ tends to $10$, and then for $a=11$, the second term becomes bigger, and the difference between the two terms will increase, as $a$ gets bigger.\\
%We may also notice that the number in the middle of the tuple plays a role: in the first tuple it is $a$, while in the second tuple it is $m'$, and at some point where $a$ becomes big enough with respect to $m'$, the first term becomes bigger than the second one.
\item For $a,a',m,n,m',n' \in \mathbb{N}^*$ with $a+m+n=a'+m'+n'$ then $amn<a'm'n'\Rightarrow \mathcal{F}(a,m,n)<\mathcal{F}(a',m',n')$.\\
\underline{Counter-example:} $\mathcal{F}(6,7,3)=835549>\mathcal{F}(4,4,8)=614823$ with $6.7.3=126<4.4.8=128$.
%However, $\mathcal{F}(6,3,7)=529957<\mathcal{F}(4,8,4)=806651$.\\
%So putting the big number in the middle of the first tuple and the small number in the middle of the second tuple is what helped to find the counter-example.\\
%But the difference should be big enough. For example: $\mathcal{F}(5,7,2)=65857<\mathcal{F}(4,4,6)=150723$ and $\mathcal{F}(5,2,7)=33606<\mathcal{F}(4,6,4)=178751$ with $5.7.2=70>4.4.6=96$. Also, $\mathcal{F}(4,6,1)=2190<\mathcal{F}(3,3,5)=6566$ and $\mathcal{F}(4,1,6)=791<\mathcal{F}(3,5,3)=8051$ with $4.6.1=24<3.3.5=45$.
\item Let $(\alpha_1,\dots,\alpha_s)$ and $(\beta_1,\dots,\beta_s)$ $\in (\mathbb{N}^*)^s$, and $\sum_{i=1}^s \alpha_i=\sum_{i=1}^s \beta_i$. If $(\beta_1,\beta_2,\dots,\beta_s)\neq (\alpha_1,\alpha_2,\dots,\alpha_s)$ and $(\beta_1,\beta_2,\dots,\beta_s)\neq(\alpha_s,\alpha_{s-1},\dots,\alpha_1)$, then $\mathcal{F}(\alpha_1,\dots,\alpha_s)\neq \mathcal{F}(\beta_1,\dots,\beta_s)$.\\
\underline{Counter-example:} $\mathcal{F}(1,3,1)=\mathcal{F}(2,1,2)=19,$ $ \mathcal{F}(1,2,3,1)=\mathcal{F}(2,1,2,2)=315,$ $ \mathcal{F}(2,4,2)=\mathcal{F}(3,2,3)=379$.
\end{itemize}

%\noindent However, we know nothing yet about the validity of many statements, as for example:
%\begin{itemize}
%\item For $a_1,...,a_s,a'_1,...,a'_s,m,n,m',n' \in \mathbb{N}^*$ we have $\mathcal{F}(a_1,\dots,a_s,m,n)<$\\$\mathcal{F}(a_1,\dots,a_s,m',n')$ $\iff (mn<m'n') \ or \ (mn=m'n', \ m<n)$.
%\item For $a,b,m,n,a',b' \in \mathbb{N}^*$, we have $(m<n) \ and \ (ab<a'b') \Rightarrow \mathcal{F}(m,a,b,n)<\mathcal{F}(m,a',b',n)$.
%\end{itemize}
\medskip
\section{Algorithmic approach of $\mathcal{F}$}

The function $\mathcal{F}$, allows us to construct a program for computing the number of oriented paths in transitive tournaments. As we will see, by comparing the different cases, we are led to state a conjecture about paths in transitive tournaments, and to discuss an interesting property about antidirected Hamiltonian paths in $TT_n$. \\
\linebreak
We first introduce in this last section the program, built using Python, to compute the values of the path-function $\mathcal{F}$.
\begin{lstlisting}[language=Python, caption={Program that computes the values of the function $\mathcal{F}$}.]

"""
This routine represents the mapping "%*\color{Strings}$\mathcal{F}$*)".
It reduces a tuple of positive integers into a single positive
integer according to the recurrence relation and the 4 properties
 atisfied by "%*\color{Strings}$\mathcal{F}$*)".
"""

def f(a):
  l = len(a)

  assert (l > 0), "f() is undefined"
  assert (l > 1 or a[0] > 0), "f(0) is undefined"

  if l == 1:
    return 1
  elif a[0] == 0:
    return f(a[1:])
  elif a[-1] == 0:
    return f(a[:-1])
  else:
    try:
      i = a.index(0)
      return f(
        a[:i - 1] +
        [a[i - 1] + a[i + 1]] +
        a[i + 2:]
      )
    except ValueError:
      # If `a` doesn't contain any zeros
      return sum(
        f(a[:i] + [a[i] - 1] + a[i + 1:])
        for i in range(l)
      )

import time
import sys

# Change the next line to your needs
a = [1, 2, 1, 1]

start_time = time.time()
ans = f(a)
end_time = time.time()

print('{} => {}'.format(a, ans))
print(
  'Took: {:.3} seconds'.format(end_time - start_time),
  file=sys.stderr
)
\end{lstlisting}

In order to investigate more about the properties of the path-function $\mathcal{F}$, we created a new program, based on the previous one, allowing us to compute all the possible values of $\mathcal{F}(a_1,\dots,a_s)$, for $1\leq s \leq p$, and $p=\sum_{i=1}^s a_i$.
%$$f(a_1,\dots,a_s), \ 1\leq s \leq p, \ p=\sum_{i=1}^s a_i$$

\begin{lstlisting}[language=Python, caption={Program that computes all values of $\mathcal{F}(a_1,\dots,a_s)$, $1\leq s \leq p$, and $p=\sum_{i=1}^s a_i$}.]

"""
Define the memorization system
"""

class memoize(dict):
  def __init__(self, f):
    self.f = f

    # Will hold some statistics
    self.total = 0
    self.miss = 0

  def __call__(self, a):
    self.total += 1
    return self[tuple(a)]

  def __missing__(self, a):
    self.miss += 1
    result = self.f(list(a))
    # Save the result for a
    self[a] = result
    # and for the reverse of a, since they are always equal
    self[a[::-1]] = result
    return result

"""
This routine represents the mapping "%*\color{Strings}$\mathcal{F}$*)".
It reduces a tuple of positive integers into a single positive
integer according to the recurrence relation and the 4 properties 
satisfied by "%*\color{Strings}$\mathcal{F}$*)".
To significantly improve the performance (speed) of this routine, 
a memorization system is used.
It allows a fast answer lookup for known tuples (already seen
before), by saving all the results in a lookup table with their 
corresponding tuples and their reverse (since they get reduced to
the same result).
"""

@memoize
def f(a):
  l = len(a)

  assert (l > 0), "f() is undefined"
  assert (l > 1 or a[0] > 0), "f(0) is undefined"

  if l == 1:
    return 1
  elif a[0] == 0:
    return f(a[1:])
  elif a[-1] == 0:
    return f(a[:-1])
  else:
    try:
      i = a.index(0)
      return f(
        a[:i - 1] +
        [a[i - 1] + a[i + 1]] +
        a[i + 2:]
      )
    except ValueError:
      # If `a` doesn't contain any zeros
      return sum(
        f(a[:i] + [a[i] - 1] + a[i + 1:])
        for i in range(l)
      )

"""
This routine generates a list of all possible tuples with 
cardinality less than or equal `p`, having the sum of all its 
elements equal `p`.
"""

def all_permutations(p):
  if p > 0:
    yield [p]
    for s in range(p - 1, 0, -1):
      for a in all_permutations(p - s):
        yield [s] + a

"""
This routine prints the results in a pretty format.
"""

def print_results(p, results):
  iw = len(str(len(results)))
  aw = 3 * p
  for (i, (a, ans)) in enumerate(results):
    print('{:>{iw}}: {:<{aw}} => {:,}'.format(
      i + 1, str(a), ans, iw=iw, aw=aw
    ))

"""
This is the main routine that combines all the previous ones.
It generates all the permutations and their corresponding
results from `f`, sorts them in ascending order and then 
prints them along with some execution statistics.
"""

def main(p):
  import time
  import sys

  start_time = time.time()
  results = [(a, f(a)) for a in all_permutations(p)]
  f_time = time.time()

  results.sort(key=lambda r: r[1])
  print_results(p, results)
  end_time = time.time()

  print(
    'Cache hit: {:.2%}'.format(1 - f.miss / f.total),
    file=sys.stderr
  )
  print(
    'Took to apply f: {:.3} seconds'.format(
      f_time - start_time
    ), 
    file=sys.stderr
  )
  print(
    'Took to sort and print: {:.3} seconds'.format(
      end_time - f_time
    ),
    file=sys.stderr
  )
  print(
    'Took in total: {:.3} seconds'.format(
      end_time - start_time
    ),
    file=sys.stderr
  )

main(6)
\end{lstlisting}
\bigskip
\medskip
\par We now give the first lists of numbers computed using the above program, giving all the possible values of $\mathcal{F}(a_1,\dots,a_s)$, $1\leq s \leq p$, for $p=\sum_{i=1}^s a_i$, and where $3\leq p \leq 7$. (The answers are in ascending order).\\

\begin{multicols}{2}
[For $p=3$:]
\begin{verbatim}
1: [3]    => 1
2: [2, 1] => 3
3: [1, 2]    => 3
4: [1, 1, 1] => 5
\end{verbatim}
\end{multicols}

\begin{multicols}{3}
[For $p=4$:]
\begin{verbatim}
1: [4]    => 1
2: [3, 1] => 4
3: [1, 3] => 4
4: [2, 2]    => 6
5: [2, 1, 1] => 9
6: [1, 1, 2] => 9
7: [1, 2, 1]    => 11
8: [1, 1, 1, 1] => 16
\end{verbatim}
\end{multicols}

\begin{multicols}{2}
[For $p=5$:]
\begin{verbatim}
 1: [5]       => 1
 2: [4, 1]    => 5
 3: [1, 4]    => 5
 4: [3, 2]    => 10
 5: [2, 3]    => 10
 6: [3, 1, 1] => 14
 7: [1, 1, 3] => 14
 8: [2, 1, 2] => 19
 9: [1, 3, 1]       => 19
10: [2, 2, 1]       => 26
11: [1, 2, 2]       => 26
12: [2, 1, 1, 1]    => 35
13: [1, 1, 1, 2]    => 35
14: [1, 2, 1, 1]    => 40
15: [1, 1, 2, 1]    => 40
16: [1, 1, 1, 1, 1] => 61
\end{verbatim}
\end{multicols}

\begin{multicols}{2}
[For $p=6$:]
\begin{verbatim}
 1: [6]          => 1
 2: [5, 1]       => 6
 3: [1, 5]       => 6
 4: [4, 2]       => 15
 5: [2, 4]       => 15
 6: [4, 1, 1]    => 20
 7: [3, 3]       => 20
 8: [1, 1, 4]    => 20
 9: [1, 4, 1]    => 29
10: [3, 1, 2]    => 34
11: [2, 1, 3]    => 34
12: [3, 2, 1]    => 50
13: [1, 2, 3]    => 50
14: [2, 3, 1]    => 55
15: [1, 3, 2]    => 55
16: [3, 1, 1, 1] => 64
17: [1, 1, 1, 3]       => 64
18: [2, 2, 2]          => 71
19: [1, 3, 1, 1]       => 78
20: [1, 1, 3, 1]       => 78
21: [2, 1, 1, 2]       => 90
22: [2, 1, 2, 1]       => 99
23: [1, 2, 1, 2]       => 99
24: [2, 2, 1, 1]       => 111
25: [1, 1, 2, 2]       => 111
26: [1, 2, 2, 1]       => 132
27: [2, 1, 1, 1, 1]    => 155
28: [1, 1, 1, 1, 2]    => 155
29: [1, 1, 2, 1, 1]    => 169
30: [1, 2, 1, 1, 1]    => 181
31: [1, 1, 1, 2, 1]    => 181
32: [1, 1, 1, 1, 1, 1] => 272
\end{verbatim}
\end{multicols}

\begin{multicols}{2}
[For $p=7$:]
\begin{verbatim}
 1: [7]          => 1
 2: [6, 1]       => 7
 3: [1, 6]       => 7
 4: [5, 2]       => 21
 5: [2, 5]       => 21
 6: [5, 1, 1]    => 27
 7: [1, 1, 5]    => 27
 8: [4, 3]       => 35
 9: [3, 4]       => 35
10: [1, 5, 1]    => 41
11: [4, 1, 2]    => 55
12: [2, 1, 4]    => 55
13: [3, 1, 3]    => 69
14: [4, 2, 1]    => 85
15: [1, 2, 4]    => 85
16: [2, 4, 1]    => 99
17: [1, 4, 2]    => 99
18: [4, 1, 1, 1] => 105
19: [1, 1, 1, 4] => 105
20: [3, 3, 1]    => 125
21: [1, 3, 3]    => 125
22: [1, 4, 1, 1] => 133
23: [1, 1, 4, 1] => 133
24: [3, 2, 2]    => 155
25: [2, 2, 3]    => 155
26: [2, 3, 2]    => 181
27: [3, 1, 1, 2] => 189
28: [2, 1, 1, 3] => 189
29: [3, 1, 2, 1] => 203
30: [1, 2, 1, 3] => 203
31: [2, 1, 3, 1] => 217
32: [1, 3, 1, 2] => 217
33: [3, 2, 1, 1]          => 245
34: [1, 1, 2, 3]          => 245
35: [2, 3, 1, 1]          => 259
36: [1, 1, 3, 2]          => 259
37: [2, 2, 1, 2]          => 315
38: [2, 1, 2, 2]          => 315
39: [1, 3, 2, 1]          => 315
40: [1, 2, 3, 1]          => 315
41: [3, 1, 1, 1, 1]       => 323
42: [1, 1, 1, 1, 3]       => 323
43: [1, 1, 3, 1, 1]       => 365
44: [1, 3, 1, 1, 1]       => 407
45: [1, 1, 1, 3, 1]       => 407
46: [2, 2, 2, 1]          => 413
47: [1, 2, 2, 2]          => 413
48: [2, 1, 1, 1, 2]       => 449
49: [2, 1, 2, 1, 1]       => 477
50: [1, 1, 2, 1, 2]       => 477
51: [2, 1, 1, 2, 1]       => 531
52: [1, 2, 1, 1, 2]       => 531
53: [2, 2, 1, 1, 1]       => 573
54: [1, 1, 1, 2, 2]       => 573
55: [1, 2, 1, 2, 1]       => 589
56: [1, 2, 2, 1, 1]       => 643
57: [1, 1, 2, 2, 1]       => 643
58: [2, 1, 1, 1, 1, 1]    => 791
59: [1, 1, 1, 1, 1, 2]    => 791
60: [1, 1, 2, 1, 1, 1]    => 875
61: [1, 1, 1, 2, 1, 1]    => 875
62: [1, 2, 1, 1, 1, 1]    => 917
63: [1, 1, 1, 1, 2, 1]    => 917
64: [1, 1, 1, 1, 1, 1, 1] => 1,385
\end{verbatim}
\end{multicols}

%\begin{multicols}{2}
%[For $p=8$:]
%\verbatiminput{f_8.txt}
%\end{multicols}
\medskip

\par Referring to the above data, and going on till $p=18$, we may observe that $\forall \ 3 \leq p \leq 18$, $\mathcal{F}(1,1,\dots,1)$ (where $1$ is repeated $p$ times) always corresponds to the maximal value between all $\mathcal{F}(a_1,\dots,a_s)$, where $1\leq s \leq p$, and $p=\sum_{i=1}^s a_i$. We also remark that $\mathcal{F}(1,2,1,\dots,1)$ is the next biggest value, and $\mathcal{F}(1,2,1,\dots,1)$ is always bigger than $\frac{\mathcal{F}(1,1,\dots,1)}{2}$. One may wonder if these properties hold for any $p$.

On the other hand, if $p$ is odd, (that is, the corresponding transitive tournament is of even order $p+1$), then the tuple $(1,1,\dots,1)$ is not symmetric. If $p$ is even, (the corresponding transitive tournament is of odd order $p+1$), then the tuple $(1,1,\dots,1)$ is symmetric, and as a consequence, $(1,2,1,\dots,1)$ is not symmetric.

Hence, using Theorem $\ref{MAIN}$, we set the following conjecture:
\begin{conjecture}
Let $TT_n$ be a transitive tournament on $n$ vertices.\\
Then $\forall$ $\alpha=(\alpha_1,\alpha_2,\dots,\alpha_s)\in \mathbb{K}_{s}$, $\sum_{i=1}^s \alpha_i=n-1$, if $n$ is even we have: $$f_{TT_n}(1,-1,1,\dots,-1,1)\geq f_{TT_n}(\alpha_1,\alpha_2,\dots,\alpha_s),$$ where $(1,-1,1,\dots,-1,1)$ has $n-1$ components, while if $n$ is odd, we have: $$f_{TT_n}(1,-2,1,-1,1,\dots,-1,1)\geq f_{TT_n}(\alpha_1,\alpha_2,\dots,\alpha_s),$$
where the number of components of $(1,-2,1,\dots,1)$ is $n-2$.
\end{conjecture}
If the conjecture is true, we can deduce that in a transitive tournament of even order, the number of antidirected Hamiltonian paths starting with a forward arc is the maximum of the numbers of oriented Hamiltonian paths starting with a forward block, for all given types.

\newpage

%\bigskip
%\medskip
%It seems the paths-function $\mathcal{F}$ holds many interesting properties that could tell a lot about the number of oriented Hamiltonian paths in a transitive tournament. We ask the following:
%\begin{problem}
%Can we find a combinatorial function $\mathcal{G}$, that allows us to compute the exact number of oriented Hamiltonian cycles of any given type in a transitive tournament?
%\end{problem}

\noindent \textbf{Acknowledgments.}
We would like to thank Ziad El Khoury Hanna for the help he provided in developing the programs, the Lebanese University for the PhD grant, and Campus France for the Eiffel excellence scholarship (Eiffel 2018).

\bigskip
\bigskip
\end{document}